\documentclass[11pts]{article}

\usepackage[margin=1in]{geometry}

\usepackage{amsmath,amssymb,latexsym,amsthm,graphicx}

\usepackage{mathrsfs}

\usepackage{enumerate}
\usepackage{url} 
\usepackage{wrapfig}
\usepackage{subfigure}
\usepackage{hyperref}
\usepackage[final]{pdfpages}

\newtheorem{theorem}{Theorem}[section]
\newtheorem{lemma}[theorem]{Lemma}

\newtheorem{defi}[theorem]{Definition}

\newcommand{\mP}{\mathcal{P}}

\newcommand{\mW}{\mathcal{W}}

\newcommand{\R}{\mathbb R}

\begin{document}
\vspace{-2cm}
\title{Sampling for the V-line Transform with Vertex in a Circle}
\author{Duy N. Nguyen\footnote{High School for the Gifted, Ho Chi Minh City, Vietnam. Email: nnduy@ptnk.edu.vn.} ~ and Linh V. Nguyen\footnote{University of Idaho, 875 Perimeter Dr, Moscow, ID 83844, USA. Email: lnguyen@uidaho.edu.}}

\maketitle
\begin{abstract}
In this paper, we consider a special V-line transform. It integrates a given function $f$ over the V-lines whose centers are on a circle centered at the origin and the symmetric axes pass through the origin. We derive two sampling scheme of the transform: the standard and interlaced ones. We prove the an error estimate for the schemes, which is explicitly expressed in term of $f$. 
\end{abstract} 
\section{Introduction}
The V-line transform arises in Compton camera imaging. Let us first recall the classical setup single-photon emission computed tomography SPECT, which is a nuclear medicine tomographic imaging technique using gamma rays.  In SPECT, weakly radioactive tracers are given to the patient. The radioactive tracers can be detected through the emission of gamma ray photons revealing the  information about biochemical processes. Then one uses a gamma camera to record photons that enter the detector surface perpendicularly. That is, the camera measures the integrals of the tracer distribution over straight lines that are orthogonal to its surface, see Fig~\ref{F:gamma}.
\begin{center}
	\begin{figure}[h]
		\begin{center}
\includegraphics[scale=0.5]{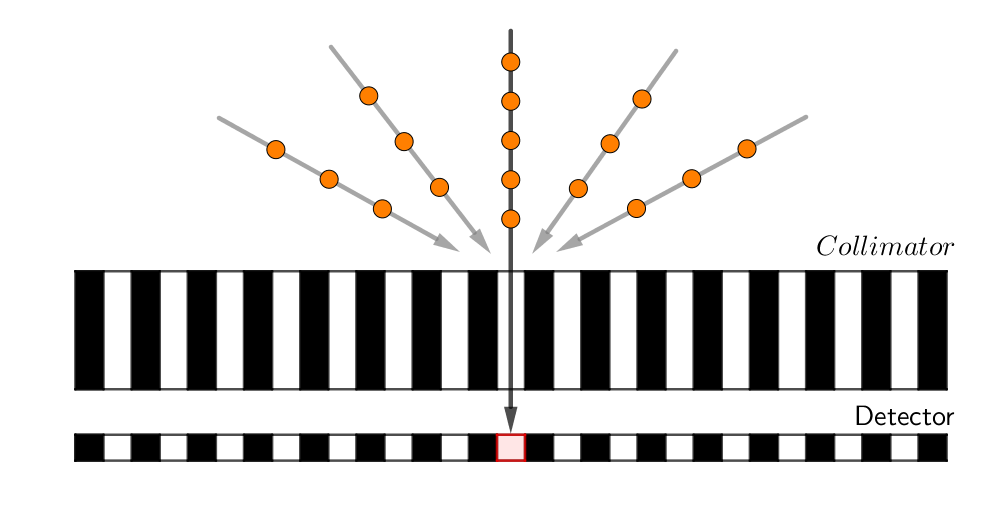}
		\end{center}
\caption{Gamma camera for SPECT. Only orthogonal photons measured in detector, other is removed.}
\label{F:gamma}
	\end{figure}
\end{center}
	This technique removes most photons and only a few photons are recorded. Therefore, a new type of camera for SPECT, which makes use of Compton scattering, was proposed by Everett \cite{everett1977gamma} and Singh \cite{singh1983electronically}. It uses electronic collimation as an alternative to mechanical collimation, which provides both high efficiency and multiple projections of the object. The camera consists of two plane gamma detectors positioned one behind the other. An emitted photon undergoes Compton scattering in the first detector surface $D_1$ and is absorbed by the second detector surface $D_2$. In each detector surface, the position and the energy of the photon are measured. The scattering angle at $D_1$ is determined via the Compton scattering formula $\cos\psi = 1-\dfrac{mc^2(E_1 – E_2)}{E_1 E_2}$, where $m$ is the electron mass, $c$ the speed of light, $E_1$ the photon energy at $D_1$, and $E_2$  the energy of the photon measured at $D_2$. So a photon observed at $x_1 \in D_1$ and $x_2 \in D_2$ with the energy $E_1,E_2$ respectively must have been emitted on the surface of the circular cone, whose vertex is at $x_1$,  central axis points from $x_2$ to $x_1$, and the half-opening angle is given by $\psi$ (see Fig.~\ref{F:Compton}).\\
\begin{center}
	\begin{figure}[h]
		\begin{center}
\includegraphics[scale=0.5]{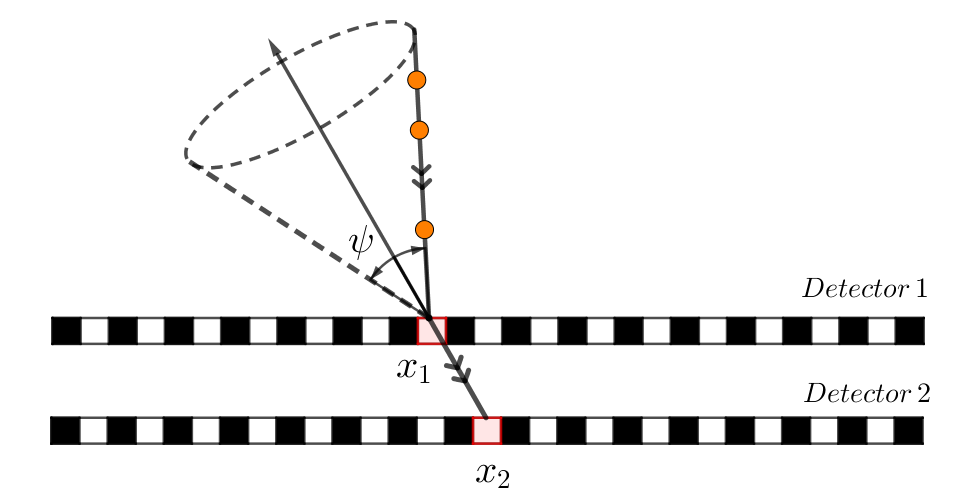}
		\end{center}
\caption{Compton camera}\label{F:Compton}
	\end{figure}
\end{center}
As a result, a Compton camera gives us the integrals of emission distribution on conical surfaces whose vertices are on the $D_1$. The mathematical problem of {\bf Compton camera imaging} is to reconstruct the emission distribution from such integrals.   

There are quite a few works on the mathematical problems of Compton camera. Specially in the two dimensional space, the cone become V-line and there 
exist some inversion formulas (e.g. \cite{basko1997analytical,morvidone2010v,truong2011new}).  In the three dimensional space, the space of conical surfaces whose vertices are on a detector surface is a five dimensional manifold. One is in the situation of redundant data. Taking advantage of such redundancy was the topic of several works (see, e.g.,\cite{terzioglu2015some}). One, however, may wish to restrict themselves into the lower dimensional data. In this article, we are interested in the sampling theory for Compton camera imaging in two dimensional space (i.e., the $V$-line transform). 


\begin{center}
 \begin{figure}[h]
	\begin{center}
	 \includegraphics[scale=0.3]{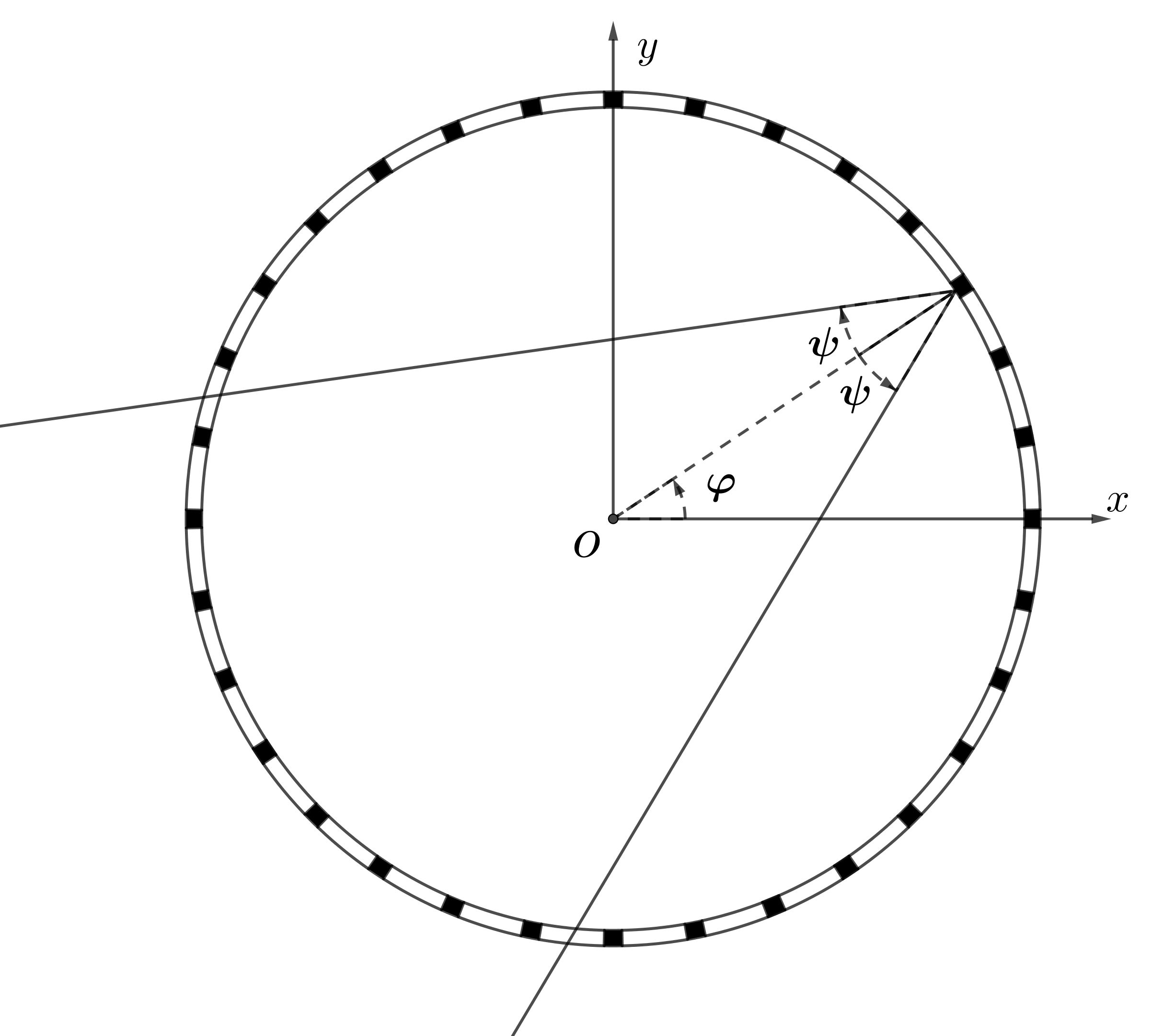}
	\end{center}	 
	\caption{Our setups: the V-line transform over all V-lines whose vertices are on the circle of radius $r$ centered at the origin and whose symmetric axis pass through the origin.}\label{V:Ours} 
 \end{figure}
\end{center}

Namely, let $f$ be a $b$-essentially band-limited function supported in the circle of radius $r_0 \leq 1$ centered at the origin. We consider the V-line transform $V(f)$ of $f$ on all the V-lines whose vertex is on the circle of radius $ r > 1$ centered at the origin and the symmetric axis passing through the origin, see Fig.~\ref{V:Ours}. 

\begin{defi} \label{V:defineVline}
	Let $f$ be a compactly supported function in $D_1(0)$. The V-line transform $Vf\left(\varphi,\psi \right)$ of $f$ is defined by
	\begin{eqnarray*}
	Vf:  \left[ 0 , 2\pi  \right)\times \left( 0;\pi \right) &\longrightarrow&   \mathbb{R},  \\
	 \left( \varphi , \psi  \right) &\longmapsto& \sum\limits_{\sigma =\pm 1}{\int\limits_{0}^{+\infty}{f\left( r \theta \left( \varphi  \right)- t \theta \left( \varphi +\sigma \psi  \right) \right)dt}}. 	
	\end{eqnarray*}
\end{defi}

We consider the problem of recovering function $Vf$ from its discrete measures $Vf\left(\varphi_k, \alpha_m \right)$. That is, we consider efficient sampling schemes to recover the V-line transform. Using Shannon sampling series and classical Fourier analysis, several authors \cite{natterer2001mathematics,natterer1993sampling,faridani2006fan,haltmeier2016sampling} study the reconstruction of the standard Radon and circular Radon transforms from its discrete measurement. More recently, Stefanov \cite{Stefanov} employed semi-classical analysis to study the same problem for generalized Radon transform. 

In this work, we introduce two sampling schemes. The first one is the standard sampling scheme, which is $(\varphi,\psi) = (\varphi_k,\psi_m)$, $0 \leq k \leq N_\varphi -1 , 0 \leq m \leq N_\psi-1$. Here, $\varphi_k$ and $\psi_m$ are evenly spaced in their domains. By restricting on the cones that intersect with the support of $f$, we only consider 
	\begin{align*}
	\varphi_k = &\dfrac{k2\pi}{N_{\varphi}}, \hspace{1cm} \text{for} \hspace{1cm} 0\leq k \leq N_{\varphi} - 1,\\
	\psi_m = &\dfrac{m2\pi}{N_{\psi}}, \hspace{1cm} \text{for} \hspace{1cm} 0 \leq m < \dfrac{N_{\psi}\arcsin(r_0 / r)}{2\pi}.
		\end{align*}
To well recover $g$ from the discrete data we need
	\begin{align*}
	N_{\varphi}\geq 2\pi rb, \quad N_{\psi} \geq 4rb.
	\end{align*}
The total number $M_0$ of needed samples has to satisfy 
	$$M_0 \geq 4(rb)^2\arcsin(r_0/r).$$
The second, more efficient, sampling scheme is given by 
		\begin{align*}
	\varphi_k = &\dfrac{k2\pi}{N_{\varphi}} \hspace{1cm} \text{for} \hspace{1cm} 0\leq k \leq N_{\varphi}-1, \\
	\psi_{k,m} = &\dfrac{\pi m}{N_{\psi}} \hspace{1cm} \text{for k, m are parity and} \hspace{1cm} 0 \leq m < \dfrac{N_{\psi}\arcsin(r_0 / r)}{\pi},	
		\end{align*}
	with the sampling conditions
	\begin{align*}
N_{\varphi}\geq 2\pi rb,\quad N_{\psi} \geq 3rb.
	\end{align*}	
The total number of samples satisfies 
		$$M_0 \geq 3(rb)^2\arcsin(r_0/r),$$ which is three quarters of the standard scheme. 
		
The article is organized as follows. In Section~\ref{S:pre}, we introduce some preliminaries. 


\section{Preliminaries} \label{S:pre}

{\bf The 2D Radon transform.} We recall the 2D Radon transform that maps a function $f$ on $\mathbb{R}^2$ into the set of its integrals over the lines on $\mathbb{R}^2$. If $\theta(\varphi)=\left(\cos\varphi,\sin\varphi\right)\in S^{1}$ is a unit vector and $s$ is a real number. The Radon transform $Rf(\varphi, s)$ is defined as
	$$Rf(\varphi, s)=\int_{x\cdot\theta=s}f(x) \, dl = \int^{+\infty}_{-\infty}{f\left(s\cos \varphi-t\sin \varphi; s\sin\varphi + t\cos\varphi \right)dt}.$$
In this notation, $Rf(\varphi,s)$ is the integral of $f$ over the line with a normal direction $\theta$ and of the distance $s$ from the origin. Let $R_{\varphi}f(\cdot)=Rf\left(\varphi, \cdot \right)$, a useful  relation between the 1D Fourier transform of $R_{\varphi}f$ and the 2D Fourier transform of $f$ is as follows
	$$\left( R_{\varphi} f \right)^{\wedge}(\sigma) = (2\pi)^{1/2}	
\hat{f}\left( \sigma \cdot \theta \right).$$
Here, the Fourier transform of a function $g$ defined on $\mathbb{R}^n$ is given by 
$$\hat{g}(\xi) = \dfrac{1}{(2\pi)^{n/2}}\int_{\mathbb{R}^n}{g(x)e^{-ix\cdot\xi}dx}.$$
The inversion formula of 2D Radon transform is
$$f(x,y)=\dfrac{-1}{4\pi^2}\int\limits_{0}^{\pi}{\int\limits_{-\infty}^{+\infty}{\dfrac{\frac{\partial}{\partial s}Rf(\varphi,s)}{s-x\cos\varphi-y\sin\varphi}ds}d\varphi}.$$
	One can see the following relationship between V-line transform and Radon transform
	$$Vf\left( \varphi ,\psi  \right)=Rf\left( \varphi +\psi -\frac{\pi }{2},r\sin \psi  \right)+Rf\left( \varphi -\psi -\frac{\pi }{2},-r\sin \psi  \right).$$
{\bf Sampling of periodic functions.}
Let $g$ be a function in $\mathbb{R}^n$ that is periodic with respect to $n$ vectors $p_1, p_2, ...,p_n $. If the matrix $P=\left( p_1, p_2, \cdots, p_n \right)$ is nonsingular then $g$ is called a P-periodic function. One denotes $L_P = P\mathbb{Z}^n$  and its reciprocal lattice $L^{\perp}_{P} = L_{2\pi P^{-T}}$.	 We define the discrete Fourier transform of $g$ be the following function on  $L^{\perp}_{P}$: 
	$$\hat{g}(\xi)=\left| \det (P) \right|^{-1}\int\limits_{P[0,1)^n}{g(x)e^{-ix\xi}dx}, \hspace{1cm} \xi \in L^{\perp}_{P}.$$
Let us note that $g$ can be recovered from its Fourier coefficients by the series \cite{terras2012harmonic}
	$$g(x)=\sum\limits_{\xi\in L^{\perp}_{P}}{\hat{g}(\xi)e^{ix\xi}}.$$
Let W be a real nonsingular $n \times n$ - matrix, our goal is study the sampling of $g$ on $L_{W}$. For this to make,  we assume $L_{P}\subset L_{W}$. This case is implied  if and only if $P=WM$ with an integer matrix $M$. Hence, $L^{\perp}_{W} \subset L^{\perp}_{P}$. Our study relies heavily on Poisson summation formula for a P-periodic function $g$, see \cite{natterer1993sampling}. It reads
	$$\sum\limits_{y \in  L_{W}/L_{P}}{g(x+y)}=\left| \dfrac{det(P)}{det(W)} \right|\sum\limits_{\xi \in L_{W}^\perp}{\hat{g}(\xi)e^{ix\xi}}.$$
Here, $L_W/L_P$ is the quotient space whose elements are of the form $y + L_P$. 	In particular, using the Poisson summation formula one obtains (see  \cite{natterer2001mathematics,faridani1991reconstructing})
	\begin{theorem} \label{D:Samplingn}
Suppose $g \in C^{\infty}\left( \mathbb{R}^n \right)$ is a P - periodic function. Let $K \subset L_{P}^{\perp} $ be a finite set such that its translates $ K+\eta$, $\eta \in L_{W}^{\perp}$ are disjoint, and $\chi_K$ denotes the characteristic function of K, i.e, $\chi_{K}(\xi)=1$ if $\xi \in K$ and $\chi_{K}(\xi)=0$ otherwise.  We define the sampling series
	$$ S_{W,K}g \left( x\right):= \left|\dfrac{ detW}{detP}\right|\sum\limits_{v \in L_W/L_P}{\widetilde{\chi_K} \left( x - v\right)g(v)}.$$
Then, 
	$$\left\| {S_{W,K}g - g} \right\|_{L^{\infty}}\leq 2 \int\limits_{L_{P}^{\perp} \backslash  K}{\left|\widehat{g}(\xi) \right|d\xi}.$$
	\end{theorem}
{\bf Sampling of the V-Line transform.} Because of $supp(f)\subset D_1(0)$, so that $Vf(\varphi, \psi)=0$ with $\psi \in \left[\dfrac{\pi}{2}, \pi\right)$. Consequently, $Vf$ can expand to an even function in $\psi$ variable and $2\pi$-periodic in both variables. We will make use of the two dimension form of Theorem \ref{D:Samplingn} to recover the V-line transform.  For $f \in {{C}^{\infty }}\left( {{D}_{1}}\left( 0 \right) \right)$, we denote $g\left( \varphi ,\psi  \right) = Vf\left( \varphi ,\psi  \right)$. Since, $g$ is $2\pi$-periodic in each variables, then the periodic matrix of $g$ is $P = 2\pi I_{2\times 2}$ and $L_{\mP}^{\perp}=\mathbb{Z}^2$. We chose matrix $W$ which satisfies the condition $L_{\mP} \subset L_{\mW}$ as  
	\[
	W = 2\pi\left(
	\begin{array}{cc}
	1/P & 0 \\
	N/(PQ) & 1/Q
	\end{array}
	\right)
	\]
where $P, Q, N$ are three integers such that $P, Q >0$ and $0 \leq N < P$. Therefor, 
 $$L_{W}/L_{P}= \{(s_j, t_{jl}) : s_j=\frac{j2\pi}{P},  t_{jl}=\frac{(l+N j/P)2\pi}{Q}, j=0,...,P-1, l=0,..., Q-1\}.$$
The sampling theorem for this setting follows as
	\begin{theorem} \label{D:Sampling}
Suppose $f \in C^{\infty}_{0}\left( D_1(0) \right)$ and $g\left( \varphi ,\psi  \right) = Vf\left( \varphi ,\psi  \right)$. Let $K \subset \mathbb{Z}^2 $ be a finite set such that its translates $ K+\eta$, $\eta \in L_{W}^{\perp}$ are disjoint. For $x \in [0,2\pi)\times[0,2\pi)$, the sampling series is
	$$ S_{W,K}g \left( x\right):= \dfrac{ 1}{PQ}\sum\limits_{v \in L_{W}/L_{P}}{\widetilde{\chi_K} \left( x - v\right)g(v)}.$$
Then, 
	$$\left\| {S_{W,K}g - g} \right\|_{L^{\infty}}\leq 2 \sum_{\mathbb{Z}^2 \backslash  K}{\left|\widehat{g}(\xi) \right|d\xi}.$$
\end{theorem}

\section{The main results}
Our goal is to use the Theorem \ref{D:Sampling} to propose some sampling conditions and derive a corresponding sampling error estimate. For this purpose,  we  assume $\hat{f}(\xi)$ be negligible for $|\xi|>b$, in the sense that the integral $\epsilon_d(f,b):=\int_{|\xi|>b}{|\xi|^{d}\, |\hat{f}(\xi)| \, d\xi}$ is small for all real number $d$. Such a function $f$ is called essentially b-band-limited. Here is the main result of this article:
	\begin{theorem} \label{T:main}
		Let $f\in C_0^{\infty}(D)$ be essentially b-band-limited $(b>1)$, and $g\left(\varphi, \psi\right)=Vf(\varphi, \psi)$. Let $\vartheta<1$ be such that $2- \vartheta^2 <r$. We define the set 
		\begin{align*}
		K=\left\{(k,m):|k|<\dfrac{rb}{\vartheta};\max\left\{|k+m|,|k-m| \right\}<rb\right\}
		\end{align*} 
in $\mathbb{R}^2$. Let $\mW$ be a real non-singular $2\times 2$ matrix such that the sets $K+2\pi\left(\mW^{-1}\right)l, \, l \in \mathbb{Z}^2$ are mutually disjoint.\\
Let
\begin{align*} &\eta_1 (\vartheta,\gamma) = \left(\frac{3}{(1-\vartheta^2)^{3/2}}  + \frac{9}{(1-\vartheta^2)^3} \frac{1}{\gamma} \right) \eta(\vartheta,\gamma), \\& \eta_2(\vartheta,\gamma) = \left(\frac{9}{(1-\vartheta^2)^{3}} + \frac{54}{(1-\vartheta^2)^{9/2}} \frac{1}{\gamma} \right) \eta(\vartheta,\gamma), \\& \eta_3(\vartheta,\gamma) =  \gamma \eta_1(\vartheta,\gamma) + \eta_2(\vartheta,\gamma),
\end{align*} where $\eta\left(\vartheta,\gamma\right)= \gamma \vartheta \exp\left(-\frac{\gamma}{3}(1-\vartheta^2)^{3/2}\right)$.  Then for $b$ being big enough 
	\begin{align*}
		\left\| {S_{W,K}g - g} \right\|_{L^{\infty}} \le \dfrac{12}{\pi}\eta^*\left(\vartheta,rb \right)\left\|f\right\|_{L^1(\mathbb{R}^2)}+\dfrac{4r^2}{\pi\vartheta^3}\left(2b+1\right)\epsilon_1(f,b),
		\end{align*}
where $\eta^*\left(\vartheta,\gamma \right)=\max\left\{ \dfrac{2b}{\vartheta^2}\eta_1\left(\vartheta,\dfrac{\gamma}{\vartheta} \right);\dfrac{1}{r}\eta_2\left(\vartheta,\dfrac{\gamma}{\vartheta^2} \right); \dfrac{\vartheta}{r^2-\bar r^2}\eta_3\left(\vartheta,\dfrac{\gamma}{\vartheta^2} \right)\right\}$. 
	\end{theorem}
In the rest of this article, we present the proof of this theorem. We then discuss two related sampling schemes.

\subsection{Proof of the main result}
We first derive some useful property of Fourier coefficient of V-line transform (see \cite{natterer2001mathematics}).
\begin{lemma} \label{L:gmk}
Suppose $f \in C^{\infty}_{0}\left( D_1(0) \right)$, $g\left( \varphi ,\psi  \right) = Vf\left( \varphi ,\psi  \right)$. Let $\widehat{g}_{k,m}$ be the Fourier coefficient of $g$. Then 
	$$\widehat{g}_{k,m}=\frac{{{i}^{k}}}{2\pi }\int\limits_{\mathbb{R}}{\int\limits_{0}^{2\pi }{\widehat{f}\left( \sigma \theta \left( \alpha  \right) \right)\left[ {{J}_{k-m}}\left( \sigma  \right)+{{J}_{k+m}}\left( \sigma  \right) \right]{{e}^{-ik\alpha }}d\alpha d\sigma }},$$ with $J_k(x)$ is Bessel function of first kind.
\end{lemma}

\begin{proof}
	We have
	\begin{eqnarray*}
	{\widehat{g}}_{k,m}&=&\frac{1}{4{{\pi }^{2}}}\int\limits_{0}^{2\pi }{\int\limits_{-\pi}^{\pi}{g\left( \varphi ,\psi  \right){{e}^{-i\left( k\varphi +m\psi  \right)}}}d\varphi d\psi} \\ 
	&=&\frac{1}{{4\pi}^{2}}\int\limits_{0}^{2\pi }\int \limits_{-\frac{\pi}{2}}^{\frac{\pi}{2}}{\left[ Rf\left( \varphi +\psi -\frac{\pi }{2}, r\theta \left( \varphi +\psi -\frac{\pi }{2} \right)\theta \left( \varphi  \right) \right) + Rf\left( \varphi -\psi -\frac{\pi }{2}, r\theta \left( \varphi -\psi -\frac{\pi}{2} \right) \theta \left( \varphi  \right) \right)\right]\times} \\
	&&{e}^{-i\left( k\varphi +m\psi  \right)}d\varphi d\psi \\ &=& I_1 + I_2,
	\end{eqnarray*}	
where $$I_{1,2} =\int\limits_{0}^{2\pi }{\int\limits_{-\frac{\pi }{2}}^{\frac{\pi }{2}}{Rf\left( \varphi + \psi -\frac{\pi }{2},r\theta \left( \varphi \pm \psi -\frac{\pi }{2} \right) \theta \left( \varphi  \right) \right) {{e}^{-im\psi }d\varphi}d\psi }}.$$ Changing variable $\alpha =\varphi +\psi -\frac{\pi }{2}$, the integral inside of $I_1$ becomes 

\begin{align*}
 \int\limits_{-\frac{\pi }{2}}^{\frac{\pi }{2}}{Rf\left( \varphi +\psi -\frac{\pi }{2},r\theta \left( \varphi +\psi -\frac{\pi }{2} \right) \theta \left( \varphi  \right) \right)} &{{e}^{-im\psi }}d\psi \\
 =&{\left( -i \right)}^{k}\int\limits_{\varphi -\pi }^{\varphi }{Rf\left[ \alpha ,r\theta \left( \alpha  \right) \theta \left( \varphi  \right) \right]{{e}^{-im\alpha }} {{e}^{im\varphi }} d\alpha }\\
 =&\frac{{{\left( -i \right)}^{k}}}{2} {{e}^{im\varphi }} \int\limits_{0}^{2\pi }{Rf\left[ \alpha ,r\theta \left( \alpha  \right) \theta \left( \varphi  \right) \right]{{e}^{-im\alpha }}d\alpha } \\
 =&\frac{{{\left( -i \right)}^{k}}}{2\sqrt{2\pi}}{{e}^{im\varphi }} \int\limits_{0}^{2\pi }{\left( \int\limits_{\mathbb{R}}{(Rf\widehat{)(}\alpha ,\sigma ) {{e}^{ir\theta \left( \alpha  \right)  \theta \left( \varphi  \right) \sigma }}d\sigma } \right)} {{e}^{-im\alpha }}d\alpha \\ 
  =&\frac{{{\left( -i \right)}^{k}}}{2} {{e}^{im\varphi }} \int\limits_{0}^{2\pi }{\int\limits_{\mathbb{R}}{\widehat{f}\left( \sigma \theta \left( \alpha  \right) \right) {{e}^{-im\alpha +ir\sigma \cos \left( \varphi -\alpha  \right)}}d\sigma d\alpha}}. 
\end{align*}

 \noindent Hence,
\begin{align*}
{{I}_{1}}=&\frac{{{\left( -i \right)}^{k}}}{8{{\pi }^{2}}} \int\limits_{0}^{2\pi }{\int\limits_{\mathbb{R}}{\widehat{f}\left( \sigma \theta \left( \alpha  \right) \right)}} \int\limits_{0}^{2\pi }{{{e}^{-i\left( k-m \right)\varphi }} {{e}^{-im\alpha +ir\sigma \cos \left( \varphi -\alpha  \right)}}d\varphi  d\sigma  d\alpha } \nonumber \\ 
=&\frac{{{\left( -i \right)}^{k}}}{8{{\pi }^{2}}} \int\limits_{0}^{2\pi }{\int\limits_{\mathbb{R}}{\widehat{f}\left( \sigma \theta \left( \alpha  \right) \right) {{e}^{-ik\alpha }}\left( \int\limits_{0}^{2\pi }{{{e}^{-i\left( k-m \right)\left( \varphi -\alpha  \right)+ir\sigma \cos \left( \varphi -\alpha  \right)}}d\varphi } \right)d\alpha d\sigma }}. \nonumber \\ 
\end{align*}
Noting that
	$$ J_{k}(x)=\dfrac{i^{-k}}{2\pi}\int_0^{2\pi}{e^{ix\cos\varphi-ik\varphi}d\varphi},$$
we obtain
\begin{align}
I_{1} = \frac{{{i}^{k}}}{4\pi } \int\limits_{0}^{2\pi }{\int\limits_{\mathbb{R}}{\widehat{f}\left( \sigma \theta \left( \alpha  \right) \right) {{e}^{-ik\alpha }}}} {{J}_{k-m}}\left( r\sigma  \right)d\alpha d\sigma . \label{eq1}
\end{align}
Similarly,
\begin{align}
I_{2}  =&\frac{i^k}{4\pi } \int\limits_{0}^{2\pi }{\int\limits_{\mathbb{R}}{\widehat{f}\left( \sigma \theta \left( \alpha  \right) \right) {{e}^{-ik\alpha }} {{J}_{k+m}}\left( r\sigma  \right)d\alpha  d\sigma }}. \label{eq2}
\end{align}
Combining (\ref{eq1}) and (\ref{eq2}), we conclude 
$$\hat{g}_{k,m}=\frac{{{i}^{k}}}{4\pi }\int\limits_{0}^{2\pi }{\int\limits_{\mathbb{R}}{\widehat{f}\left( \sigma \theta \left( \alpha  \right) \right)\left[ {{J}_{k-m}}\left( r\sigma  \right)+{{J}_{k+m}}\left( r\sigma  \right) \right] {{e}^{-ik\alpha }}d\alpha d\sigma }}.\label{eq3}$$
\end{proof}


In order to estimate $\hat g_{m,k}$ we will need the following result 
\begin{lemma} \label{L:int}
Suppose $f \in C^{\infty}_{0}\left( D_1(0) \right)$, $g\left( \varphi ,\psi  \right) = Vf\left( \varphi ,\psi  \right)$. Let $\widehat{g}_{k,m}$ be the Fourier coefficient of $g$. Then 
	$$\Big| \int\limits_{\mathbb{R}} \int\limits_{0}^{2\pi }{\widehat{f}\left( \sigma \theta \left( \alpha  \right) \right) {{J}_{l}}\left( \sigma  \right) {{e}^{-ik\alpha }}d\alpha d\sigma } \Big| \leq\int_{D_1(0)}{|f(x)|\left|\int^{\sigma}_{-\sigma}{J_{l}(r\sigma)J_k(\sigma|x|)d\sigma}\right|dx} + 2 \epsilon_{-1}(f,\sigma),$$

\end{lemma}


\begin{proof}

Since the Bessel function is bounded by $1$, 
\begin{align*} \Big| \int\limits_{\mathbb{R}} \int\limits_{0}^{2\pi }{\widehat{f}\left( \sigma \theta \left( \alpha  \right) \right) {{J}_{l}}\left( \sigma  \right) {{e}^{-ik\alpha }}d\alpha d\sigma } \Big|
 & \leq \dfrac{1}{4\pi}\left|\int^{\sigma}_{-\sigma}\int^{2\pi}_0\hat{f}(\sigma\theta)e^{-ik\alpha}J_{l}(r\sigma)d\alpha d\sigma \right|+\dfrac{1}{2\pi}\epsilon_{-1}(f,\sigma).\end{align*} 
Expressing $\hat{f}(\sigma \theta)$ by its definition, we obtain
\begin{align*}
	\int\limits_{0}^{2\pi }{\widehat{f}\left( \sigma \theta \left( \alpha  \right) \right){{e}^{-ik\alpha }}d\alpha }=&\dfrac{1}{2\pi}\int^{2\pi}_0\int_{D_1(0)}f(x)e^{-i\sigma\theta x}e^{-ik\alpha}dxd\alpha\\
	=&\dfrac{1}{2\pi}\int_{D_1(0)}f(x)\int^{2\pi}_0e^{-i\left(\alpha|x|\cos (\alpha-\psi)+k(\alpha-\psi)\right)}e^{-ik\psi}d\alpha dx\\
	=&\int_{D_1(0)}f(x)J_{k}(-\sigma|x|)dx.
\end{align*}
This gives us the desired in equality. 
	
\end{proof} 	
The following inequality in \cite{siegel1953inequality}
$$ J_v(vs) \leq \dfrac{s^v \exp\left(v(1-s^2)^{1/2}\right))}{\left(1+(1-s^2)\right)^{1/2}}, \hspace{1cm} v\geq 0,~ 0<s\leq 1,$$
yields 
$$ J_v(vs) \leq e^{-\frac{s}{3} (1-s^2)^{3/2}}, \hspace{1cm} v\geq 0,\, 0<s \leq 1.$$
We obtain
	\begin{align}  \sup\limits_{|x|\leq 1}\int\limits_{-\vartheta k}^{\vartheta k}{\left| J_k(\sigma|x| )\right|d\sigma} & \leq 2 k \vartheta\sup\limits_{|x|\leq 1}\int\limits_{0}^{1}J_k(k\sigma \vartheta |x|) d\sigma  \leq 2 k \vartheta \int\limits_{0}^{1} e^{-\frac{\gamma}{3} (1-|\vartheta \sigma |^2)^{3/2}} d\sigma  \leq 2\eta(\vartheta,\gamma),\label{E:eta1} \end{align} 

where
$$\eta\left(\vartheta,\gamma\right)= k\vartheta \exp\left(-\frac{\gamma}{3}(1-\vartheta^2)^{3/2}\right).$$

\begin{lemma}\label{L:estint} Assume that $|k|< \vartheta^{-1}|l|$. Choosing $\sigma =\vartheta|l|$, then the following holds. 
	\begin{align*}
	\int_{D_1(0)}{|f(x)|\left|\int^{\sigma}_{-\sigma}{J_{l}(r\sigma)J_k(\sigma|x|)d\sigma}\right|dx}\le 2\eta\left(\vartheta,|l|\right)\left\|f\right\|_{L^1(\mathbb{R}^2)}.
	\end{align*}
\end{lemma}

\begin{proof}
Indeed, using the fact that $|J_k(s)| \leq 1$ for all $s \geq 1$, we obtain
	\begin{align*}
	\int_{D_1(0)}{|f(x)|\left|\int^{\sigma}_{-\sigma}{J_{l}(r\sigma)J_k(\sigma|x|)d\sigma}\right|dx}\le &\left(\int_{|\sigma|\le \vartheta|l|/r}\left|J_{l}(r\sigma )\right|d\sigma\right) \, \|f\|_{L^1(\mathbb{R}^2)}\nonumber\\
	\le& \left(\dfrac{1}{r}\int_{|\sigma|\le \vartheta|l|}\left|J_{l}(\sigma )\right|d\sigma\right) \left\|f\right\|_{L^1(\mathbb{R}^2)}\nonumber\\
	\stackrel{(\ref{E:eta1})}{\le} & \dfrac{2}{r}\eta\left(\vartheta,|l|\right)\left\|f\right\|_{L^1(\mathbb{R}^2)}. 
	\end{align*}
%
\end{proof} 
We will also need the following result
	\begin{lemma} \label{L:pala}
Suppose $f \in C^{\infty}_{0}\left( D_1(0) \right)$ and $\bar r:= 2 -\vartheta^2 <r$, then for $|k|\geq \vartheta^{-1}|l|$,
$$\Big| \int\limits_{\mathbb{R}} \int\limits_{0}^{2\pi }{\widehat{f}\left( \sigma \theta \left( \alpha  \right) \right) {{J}_{l}}\left( \sigma  \right) {{e}^{-ik\alpha }}d\alpha d\sigma } \Big| \leq \dfrac{1}{2\pi (r^2 - \bar r^2) k \vartheta} \eta(k,\vartheta) \|f\|_{L^1(\R^2)}.$$ 

\end{lemma}

\begin{proof}
 From \cite{palamodov1995localization}, for any positive real number $\epsilon$ such that $\rho:=\sqrt{\cosh 2\epsilon}\leq r$, we have
	\begin{align*}
	\Big| \int\limits_{\mathbb{R}} \int\limits_{0}^{2\pi }{\widehat{f}\left( \sigma \theta \left( \alpha  \right) \right) {{J}_{l}}\left( \sigma  \right) {{e}^{-ik\alpha }}d\alpha d\sigma } \Big|
	\leq& \dfrac{I(f)}{2\pi}\exp\left( -\epsilon|k|+\frac{\delta}{r}|l|\right)
	\end{align*} where $\delta:=\int_{0}^{\epsilon}{\sqrt{\cosh2t} \,dt}$ and
$$I(f):=\int_{D_1(0)}{\dfrac{\left| f(x,y)\right|}{r^2-\bar r^2(x^2+y^2)}dxdy}.$$
Simple calculation give for for  all $\epsilon<1$, $$\delta=\int_{0}^{\epsilon}{\sqrt{\cosh2t} \,dt} \leq \epsilon \sqrt{\cosh 2\epsilon} < \epsilon(1+\epsilon).$$  
Choosing $\epsilon = \left(1-\vartheta^2 \right)$ we obtain
	\begin{equation*}
	\epsilon -\dfrac{\delta \vartheta}{r} \geq \epsilon (1- \frac{1+ \epsilon}{r} \vartheta) \geq \epsilon(1-\vartheta) \geq  \epsilon(1-\sqrt{1-\epsilon}) \geq  \frac{1}{3}\epsilon^{3/2} = \left(1-\vartheta^2 \right)^{3/2}
	\end{equation*}
Therefore, for $|k|\geq \vartheta^{-1}|l|$,
	\begin{equation*}
	\exp\left( -\epsilon|k|+\frac{\delta}{r}|l|\right) \leq \exp\left(-\frac{|k|}{3}(1-\vartheta^2)^{3/2}\right).
	\end{equation*}		
Moreover, since $\rho = \cosh (2\epsilon) < 1+ \epsilon = \bar r$, 
$$I(f):=\int_{D_1(0)}{\dfrac{\left| f(x,y)\right|}{r^2-\bar r^2(x^2+y^2)}dxdy}\leq \frac{1}{r^2 - \bar r^2} \|f\|_{L_1(\R^2)}.$$
The above two inequalities finish our proof. 
\end{proof}		
%
	
We are now ready to prove Theorem~\ref{T:main}. 
\begin{proof}[\bf Proof of Theorem~\ref{T:main}]	 Due to Theorem~\ref{D:Sampling} for the solution, suffices to prove that
	$$2 \sum_{\mathbb{Z}^2 \backslash  K} \left|\widehat{g}_{k,m} \right| \leq \dfrac{12}{\pi}\eta^*\left(\vartheta,rb \right)\left\|f\right\|_{L^1(\mathbb{R}^2)}+\dfrac{4r^2}{\pi\vartheta^3}\left(2b+1\right)\epsilon_1(f,b).$$
	
Indeed, from Lemma~\ref{L:gmk}, we obtain
	$$|\widehat{g}_{k,m}|\leq \frac{1}{2\pi} \bigg|\int\limits_{\mathbb{R}} \int\limits_{0}^{2\pi } \widehat{f}( \sigma \theta\left( \alpha  \right)){{J}_{k-m}}\left( \sigma  \right) \, e^{-ik\alpha } \, d\alpha d\sigma \bigg| + \frac{1}{2\pi} \bigg |\int\limits_{\mathbb{R}} \int\limits_{0}^{2\pi} \widehat{f}(\sigma \theta \left( \alpha  \right)) J_{k+m}\left( \sigma  \right) \, e^{-ik\alpha} \, d\alpha d\sigma \bigg |.$$ 

Therefore,
\begin{align*} \sum_{\mathbb{Z}^2 \backslash  K} \left|\widehat{g}_{k,m} \right| \leq S_1 + S_2,\end{align*} 
where
\begin{align*} S_{1,2} =  \sum_{\mathbb{Z}^2 \backslash  K}  \frac{1}{2\pi} \bigg|\int\limits_{\mathbb{R}} \int\limits_{0}^{2\pi } \widehat{f}( \sigma \theta\left( \alpha  \right)){{J}_{k \mp m}}\left( \sigma  \right) \, e^{-ik\alpha } \, d\alpha d\sigma \bigg|.\end{align*}


Let us denote
	\begin{align*}
	\bar \eta_1\left(\vartheta,\gamma\right)=\sum\limits_{m>k}{\eta\left(\vartheta,m\right)},\quad
	\bar \eta_2\left(\vartheta,\gamma\right)=\sum\limits_{m>k}{\bar \eta_1\left(\vartheta,m\right)}, \quad
	\bar \eta_3\left(\vartheta,\gamma\right)=\sum\limits_{m>k}{m \eta\left(\vartheta,m\right)},	
	\end{align*}
We note that $\eta\left(\vartheta,\gamma\right)$ exponentially decays as $k \to \infty$ and 
$$\sum\limits_{m>k}{m^d\eta\left(\vartheta,m\right)}\leq \int\limits_{k}^\infty {s^d\eta\left(\vartheta,s\right)} = \int_{k}^{\infty}{\vartheta s^{d+1} \exp\left(-\frac{s}{3}(1-\vartheta^2)^{3/2}\right)ds},$$
for $d \geq -1$ and $k$ big enough. Direct calculations then show
	\begin{align}
	\bar \eta_1\left(\vartheta,\gamma\right)&\leq \left(\frac{3}{(1-\vartheta^2)^{3/2}}+ \frac{9}{(1-\vartheta^2)^3} \frac{1}{\gamma}\right) \eta(\vartheta,\gamma) = \eta_1(\vartheta,\gamma),\label{eq4}\\
	\bar \eta_2\left(\vartheta,\gamma\right)& \leq \left(\frac{9}{(1-\vartheta^2)^{3}}+ \frac{54}{(1-\vartheta^2)^{9/2}} \frac{1}{\gamma}\right) \eta(\vartheta,\gamma) = \eta_2(\vartheta,\gamma),\label{eq5}\\
	\bar \eta_3\left(\vartheta,\gamma\right) & \leq  \gamma \bar \eta_1(\vartheta,\gamma) + \bar \eta_2(\vartheta,\gamma) = \eta_3(\vartheta,\gamma). \label{eq6}
	\end{align}

\begin{center}
	\begin{figure}[h]
 		\begin{center}
\includegraphics[width=0.3\textwidth]{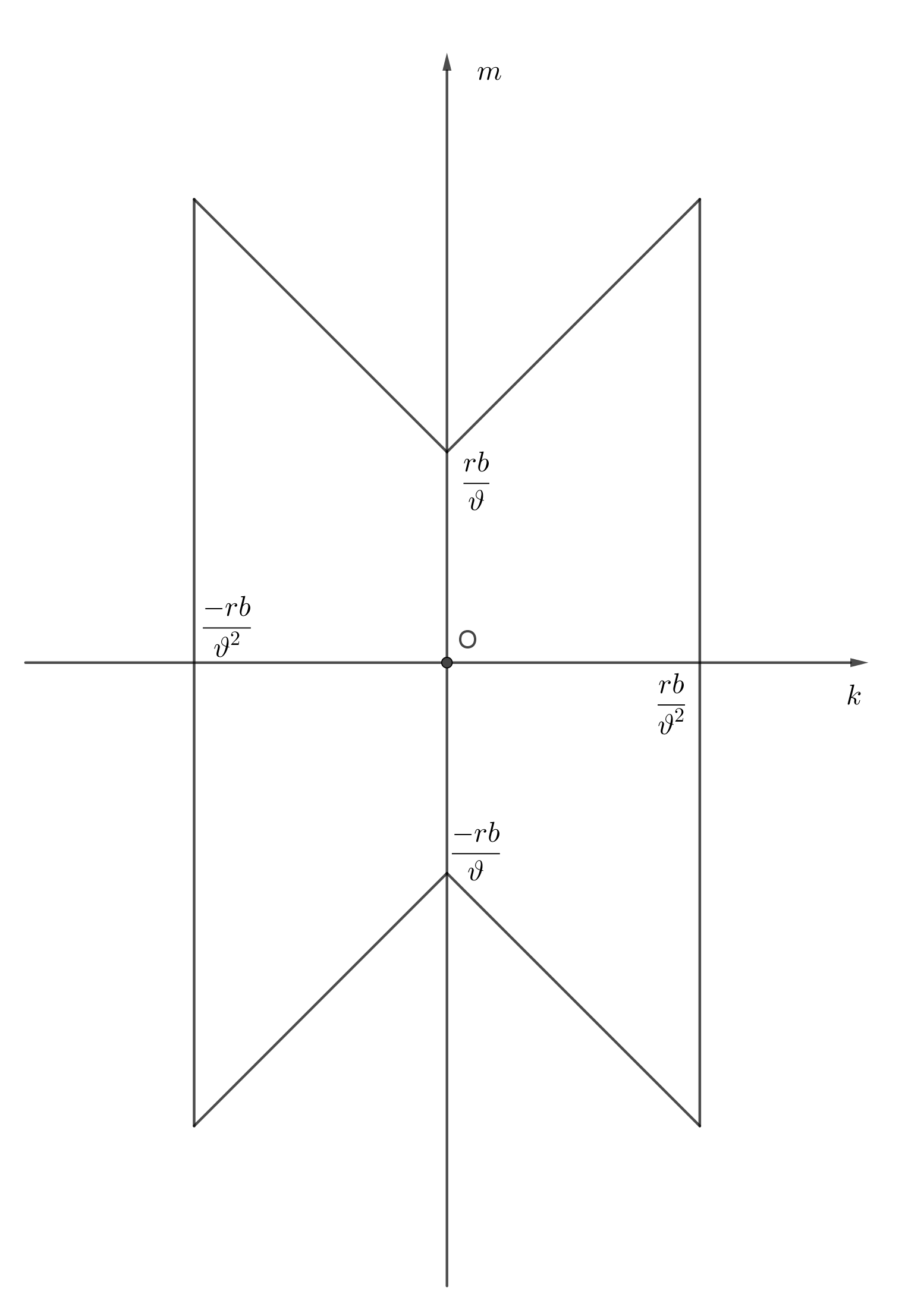}
		 \end{center}
 \caption{\label{F:Ours} The set K for $\vartheta = \frac{5}{6}, b=5, r=\frac{3}{2} $.}
	\end{figure}
\end{center}

\underline{{\bf PART 1}: Estimate $S_1$}

We decompose $S_1 = S_{11} + S_{12}+S_{13}$ where each $S_{ij}$ is the sum ranging over the region $\Sigma_{ij}$, where 
		\begin{eqnarray*}
		\Sigma_{11}:& =&\left\{(k,m) \in \mathbb{Z}^2: |k|<\dfrac{rb}{\vartheta^2}; |k-m| \geq\dfrac{rb}{\vartheta}\right\}, \\
		\Sigma_{12}:&= &\left\{(k,m) \in \mathbb{Z}^2: |k|\ge \dfrac{rb}{\vartheta^2}; |k| > \dfrac{|k-m|}{\vartheta} \right\},\\
		\Sigma_{13}:&=&\left\{(k,m) \in \mathbb{Z}^2: |k|\ge \dfrac{rb}{\vartheta^2}; |k| \leq \dfrac{|k-m|}{\vartheta}\right\}.\\
		\end{eqnarray*}
\begin{center}
	\begin{figure}[h]
		\begin{center}
\includegraphics[width=0.3\textwidth]{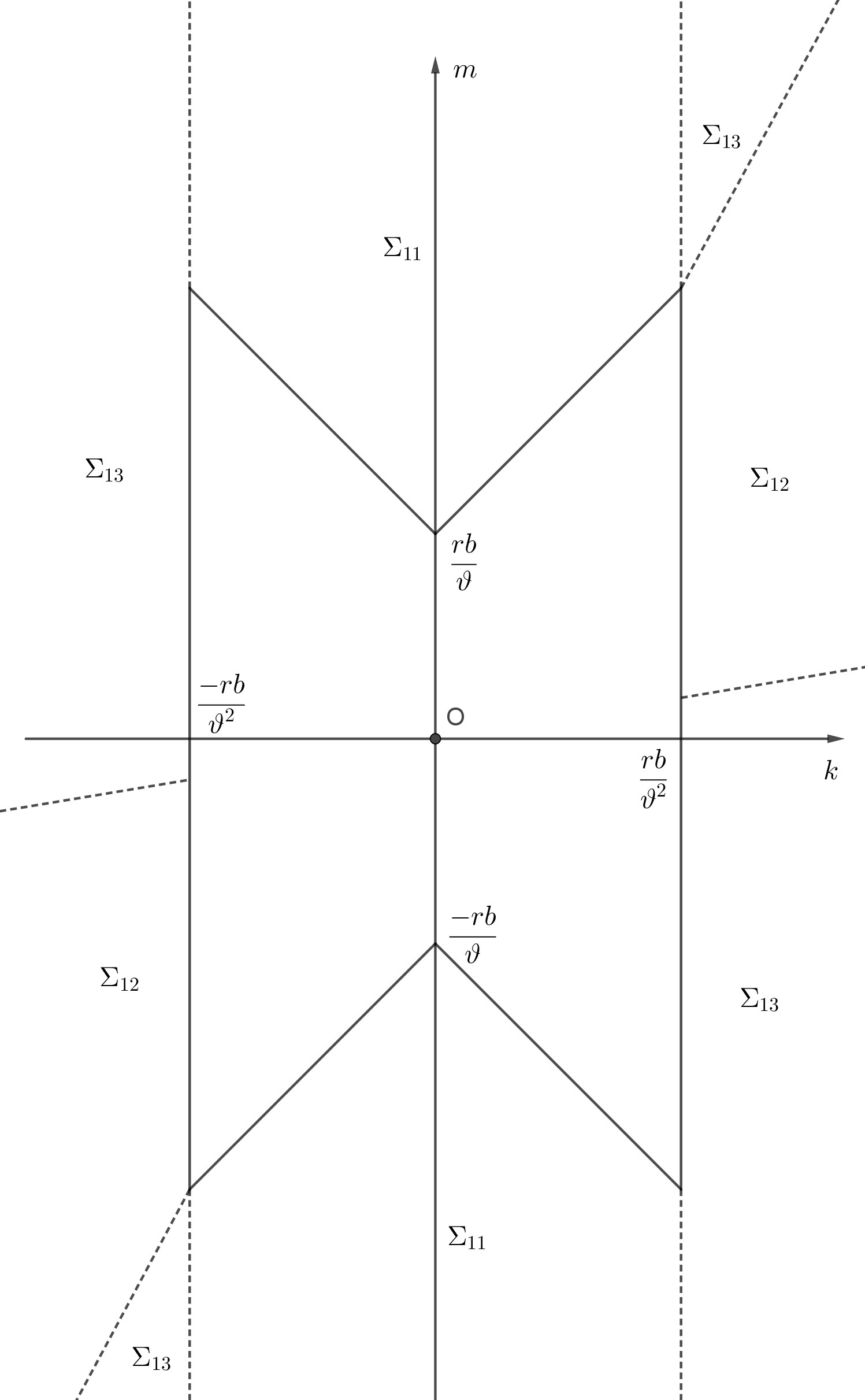}
		\end{center}
\caption{\label{F1:Ours} The parts of set $\mathbb{Z}^2 \backslash K$ when we calculate $I_1$.}
	\end{figure}
\end{center}
\begin{itemize}
\item  For the sum $S_{11}$, we note that $|k-m|\ge \vartheta |k|$ in $\Sigma_{11}$. Using Lemma~\ref{L:int} and Lemma~\ref{L:estint}, we obtain
		\begin{align*}
		\Sigma_{11} \le &\sum_{|k|\le \dfrac{rb}{\vartheta^2}} \; \sum_{|k-m|\ge \dfrac{rb}{\vartheta}}\left(\dfrac{1}{2r\pi}\eta\left(\vartheta,|k-m|\right)\left\|f\right\|_{L^1(\mathbb{R}^2)}+\dfrac{1}{2\pi}\epsilon_{-1}\left(f,\dfrac{\vartheta|k-m|}{r}\right)\right)\\
		\le &\sum_{|k|\le \dfrac{rb}{\vartheta^2}} \; \sum_{l\ge \dfrac{rb}{\vartheta}}\left(\dfrac{1}{r\pi}\eta\left(\vartheta,l\right)\left\|f\right\|_{L^1(\mathbb{R}^2)}+\dfrac{1}{\pi}\epsilon_{-1}\left(f,\dfrac{\vartheta l}{r}\right)\right)\\
		\stackrel{(\ref{eq4})}{\le} &\dfrac{2rb}{\vartheta^2}\left(\dfrac{1}{\pi r}\eta_1\left(\vartheta,\dfrac{rb}{\vartheta}\right)\left\|f\right\|_{L^1(\mathbb{R}^2)}+\dfrac{r}{\pi \vartheta}\epsilon_{0}(f,b)\right)\\
		\le &\dfrac{2b}{\pi\vartheta^2}\eta_1\left(\vartheta,\dfrac{rb}{\vartheta}\right)\left\|f\right\|_{L^1(\mathbb{R}^2)}+\dfrac{2r^2b}{\pi \vartheta^3}\epsilon_{0}(f,b),
		\end{align*}
where we have used, $0<\mu$ and $b>1$ (see, e.g., \cite{natterer2001mathematics})
\begin{equation}
\sum\limits_{l\geq b/\mu}{\epsilon_d\left(f,\mu l\right)\leq \dfrac{1}{\mu}\epsilon_{d+1}\left(f,b\right)}, \label{eq7}
\end{equation}
for $\mu = \vartheta/r$. 		
		
\item For the sum $S_{12}$, we notice that $|k-m| < \vartheta |k|$. Using Lemma~\ref{L:pala},
		\begin{align*}
		S_{12} \le &\sum_{|k|\ge \dfrac{rb}{\vartheta^2}}\dfrac{\vartheta|k|}{2\pi\left(r^2-\bar r^2 \right)}\eta\left(\vartheta,|k|\right)\left\|f\right\|_{L^1(\mathbb{R}^2)}
		\le 	\dfrac{\vartheta}{\pi\left(r^2-\bar r^2 \right)}\sum_{l\ge \dfrac{rb}{\vartheta^2}}l \, \eta\left(\vartheta, l\right)\left\|f\right\|_{L^1(\mathbb{R}^2)}\\
		\stackrel{(\ref{eq6})}{\le} &\dfrac{\vartheta}{\pi\left(r^2-\bar r^2 \right)}\eta_3\left(\vartheta,\dfrac{rb}{\vartheta^2}\right)\left\|f\right\|_{L^1(\mathbb{R}^2)}.
		\end{align*}
\item	In $\Sigma_{13}$, we get the sum for the $m$ first. Hence, from Lemma~\ref{L:estint}
		\begin{align*}
		\Sigma_{13} \le &\sum_{|k|\ge \dfrac{rb}{\vartheta^2}} \; \sum_{|k-m|\geq \vartheta|k|}\left(\dfrac{1}{2\pi r}\eta\left(\vartheta, |k-m|\right)\left\|f\right\|_{L^1(\mathbb{R}^2)}+\dfrac{1}{2\pi}\epsilon_{-1}\left(f,\dfrac{\vartheta|k-m|}{r}\right)\right)\\
		\stackrel{(\ref{eq4}),(\ref{eq7})}{\le} &\sum_{|k|\ge \dfrac{rb}{\vartheta^2}}\left(\dfrac{1}{2\pi r}\eta_1\left(\vartheta, \vartheta|k|\right)\left\|f\right\|_{L^1(\mathbb{R}^2)}+\dfrac{r}{2\pi\vartheta}\epsilon_0\left(f,\dfrac{\vartheta^2 |k|}{r}\right)\right)\\
		\le &\sum_{l\ge \dfrac{rb}{\vartheta^2}}\left(\dfrac{1}{\pi r}\eta_1\left(\vartheta, l\right)\left\|f\right\|_{L^1(\mathbb{R}^2)}+\dfrac{r}{\pi\vartheta}\epsilon_0\left(f,\dfrac{\vartheta^2 l}{r}\right)\right)\\
		\stackrel{(\ref{eq5}),(\ref{eq7})}{\le} &\dfrac{1}{\pi r}\eta_2\left(\vartheta,\dfrac{rb}{\vartheta^2}\right)\left\|f\right\|_{L^1(\mathbb{R}^2)}+\dfrac{r^2}{\pi\vartheta^3}\epsilon_{1}(f,b).
		\end{align*}
\end{itemize}

\underline{{\bf PART 2}: Estimate $S_2$}

Similarly, we consider the sums $S_{21}, S_{22}, S_{23}$ over the regions
		\begin{align*}
		\begin{matrix}
		&\Sigma_{21}:&|k|<\dfrac{rb}{\vartheta^2}&;&|k+m| \geq\dfrac{rb}{\vartheta}\\
		&\Sigma_{22}:&|k|\ge \dfrac{rb}{\vartheta^2}&;& |k| > \dfrac{|k+m|}{\vartheta} \\
		&\Sigma_{23}:&|k|\ge \dfrac{rb}{\vartheta^2}&;& |k| \leq \dfrac{|k+m|}{\vartheta}.\\
		\end{matrix}
		\end{align*}
\begin{center}
	\begin{figure}[h]
		\begin{center}
\includegraphics[width=0.3\textwidth]{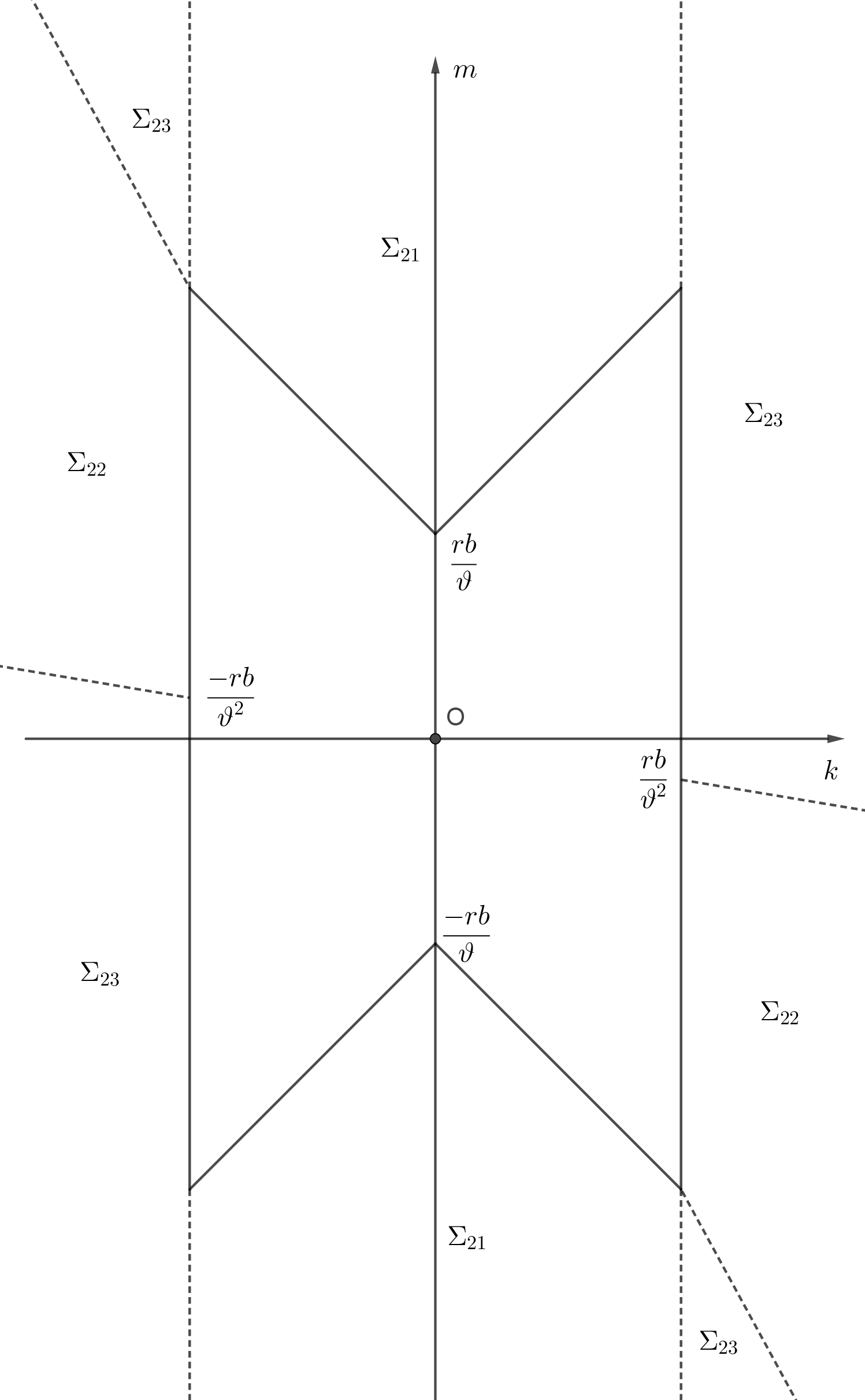}
		\end{center}
\caption{\label{F2:Ours} The parts of set $\mathbb{Z}^2 \backslash K$ when we calculate $I_2$.}
	\end{figure}
\end{center}
Then, 
		\begin{align*}
		S_{21} \le &\dfrac{2b}{\pi\vartheta^2}\eta_1\left(\vartheta,\dfrac{rb}{\vartheta}\right)\left\|f\right\|_{L^1(\mathbb{R}^2)}+\dfrac{2r^2 b}{\pi\vartheta^3}\epsilon_{0}(f,b),\\
		S_{22} \le &\dfrac{\vartheta}{\pi\left(r^2-\bar r^2 \right)}\eta_3\left(\vartheta,\dfrac{rb}{\vartheta^2}\right)\left\|f\right\|_{L^1(\mathbb{R}^2)},\\
		S_{23} \le &\dfrac{1}{\pi r}\eta_2\left(\vartheta,\dfrac{rb}{\vartheta^2}\right)\left\|f\right\|_{L^1(\mathbb{R}^2)}+\dfrac{r^2}{\pi\vartheta^3}\epsilon_{1}(f,b).
		\end{align*}
		
\noindent\underline{{\bf FINISHING THE PROOF}}

Combining the estimates in \textbf{PART 1} and \textbf{PART 2}, we obtain
	\begin{align*}
		2 \sum_{(k,m)\notin K}|\hat{g}(k,m)|\le \dfrac{12}{\pi}\eta^*\left(\vartheta,rb \right)\left\|f\right\|_{L^1(\mathbb{R}^2)}+\dfrac{4r^2}{\pi\vartheta^3}\left(2b+1\right)\epsilon_1(f,b),
		\end{align*}
	where $\eta^*\left(\vartheta,rb \right)=\max\left\{ \dfrac{2b}{\vartheta^2}\eta_1\left(\vartheta,\dfrac{rb}{\vartheta} \right);\dfrac{1}{r}\eta_2\left(\vartheta,\dfrac{rb}{\vartheta^2} \right); \dfrac{\vartheta}{r^2-\bar r^2}\eta_3\left(\vartheta,\dfrac{rb}{\vartheta^2} \right)\right\}$\\
According to Theorem \ref{D:Sampling}, we conclude
	\begin{align*}
		\left\| {S_{W,K}g - g} \right\|_{L^{\infty}} \le \dfrac{12}{\pi}\eta^*\left(\vartheta,rb \right)\left\|f\right\|_{L^1(\mathbb{R}^2)}+\dfrac{4r^2}{\pi\vartheta^3}\left(2b+1\right)\epsilon_1(f,b).
		\end{align*}
\end{proof}

\subsection{Sampling schemes for $g(\varphi, \psi)$}
In this section, we consider two schemes that satisfy the conditions in Theorem~\ref{T:main}. The first one is the standard scheme, where the sampling locations is the Cartesian product. The second one, more efficient, is an interlaced scheme. \\

\noindent{\bf Standard Sampling Scheme}
\begin{center}
	\begin{figure}[h]
		\begin{center}
\includegraphics[width=0.35\textwidth]{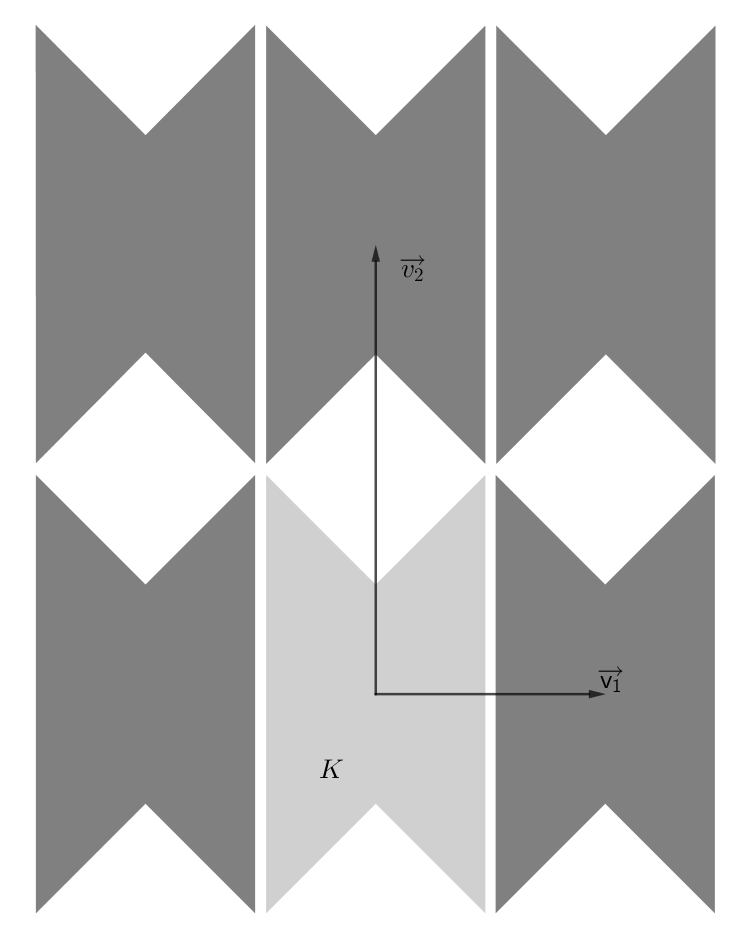}
		\end{center}
\caption{\label{F3:Ours} Geometry of the standard sampling scheme.}
	\end{figure}
\end{center}
	For the standard sampling, we chose $2\pi W^{-T}$ such that the translates $K+2\pi W^{-T}m$ are disjoint for any $m\in \mathbb{Z}^2$. From Figure \ref{F3:Ours} we have a choice
	\[
	 2\pi W^{-T}=\left(
\begin{array}{cc}
\dfrac{2rb}{\vartheta^2} & 0 \\
0 & \dfrac{2rb}{\vartheta}\left(1+\dfrac{1}{\vartheta}\right)
\end{array}
\right)
	\]
So that
		\[
	 W=\left(
\begin{array}{cc}
\pi \vartheta^2 / rb & 0 \\
0 & \dfrac{\pi\vartheta^2}{rb\left(1+\vartheta\right)}
\end{array}
\right)
	=:\left(
\begin{array}{cc}
2\pi / N_{\varphi} & 0 \\
0 & 2\pi / N_{\psi}
\end{array}
\right)
	\]
	We assume that $N_{\varphi}$ and $N_{\psi}$ are integers, otherwise we replace them by $[N_{\varphi}]$ and $[N_{\psi}]$.\\
	Assume that $f$ is supported in $|x|< r_0 \leq 1$, so $m$ can be restricted to $|m|< \dfrac{N_{\psi}\arcsin(r_0 / r)}{2\pi}$. Because the function $g\left(\varphi, \psi\right)$ is even in $\psi$ so we only chose $\psi \geq 0$. This yields the standard detector system
$g_{k,m} = g\left(\varphi_k, \psi_m \right)$, where 
	\begin{align*}
	\varphi_k = &\dfrac{k2\pi}{N_{\varphi}}, \hspace{1cm} \text{for} \hspace{1cm} 0\leq k \leq N_{\varphi} - 1\\
	\psi_m = &\dfrac{m2\pi}{N_{\psi}}, \hspace{1cm} \text{for} \hspace{1cm} 0 \leq m < \dfrac{N_{\psi}\arcsin(r_0 / r)}{2\pi}.
		\end{align*}
The sampling conditions in Theorem~\ref{T:main} are satisfied if
\begin{align*}
N_{\varphi}\geq &\dfrac{2\pi rb}{\vartheta^2}, \quad N_{\psi} \geq \dfrac{2rb\left(1+\vartheta\right)}{\vartheta^2}.\\
\end{align*}

\noindent{\bf Efficient Sampling Scheme}
\begin{center}
	\begin{figure}[h]
		\begin{center}
\includegraphics[width=0.35\textwidth]{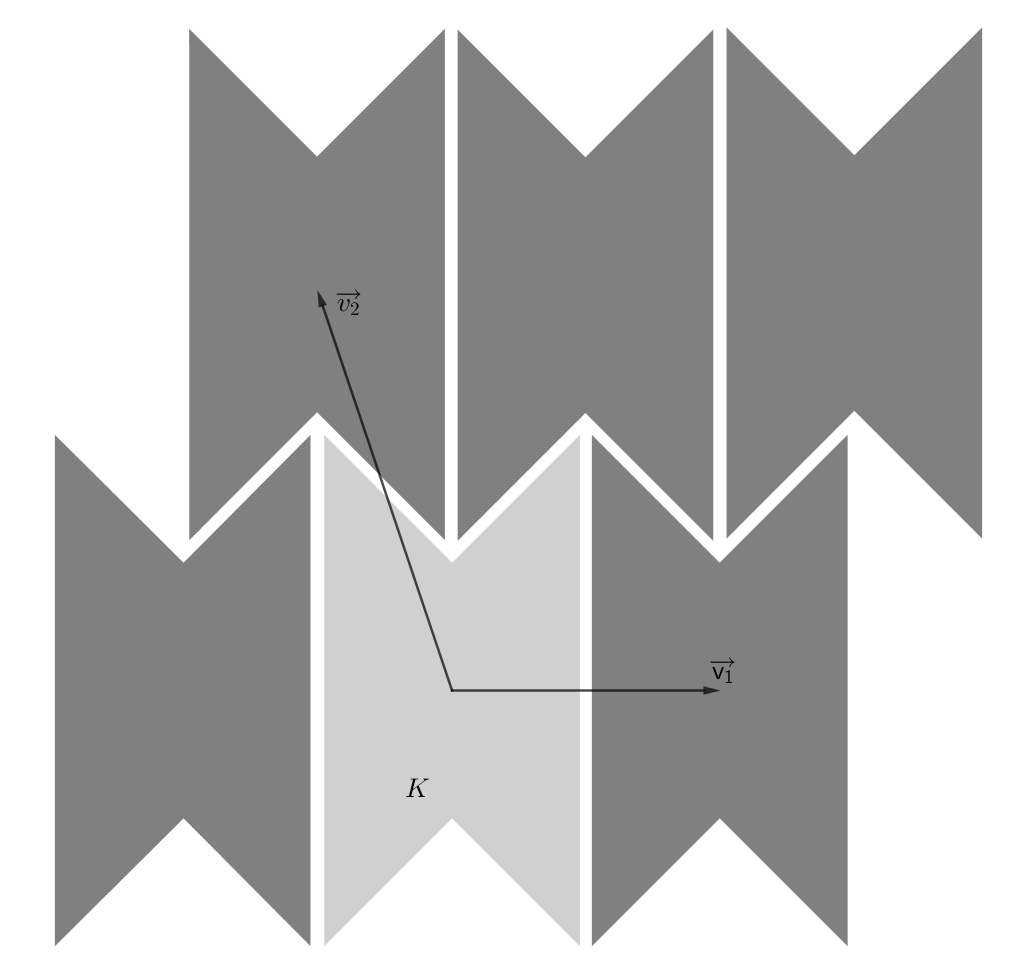}\\
		\end{center}
\caption{\label{F4:Ours} Geometry of efficient sampling scheme.}
	\end{figure}
\end{center}
Again, we need to chose $2\pi W^{-T}$ such that the translates $K+2\pi W^{-T}m$ are disjoint for any $m\in \mathbb{Z}^2$. Making the following choice (see Fig.~\ref{F4:Ours})
	\[
	 2\pi W^{-T}=\left(
\begin{array}{cc}
\dfrac{2rb}{\vartheta^2} & -\dfrac{rb}{\vartheta^2} \\
0 & \dfrac{rb}{\vartheta}\left(2+\dfrac{1}{\vartheta}\right)
\end{array}
\right),
	\]
We obtain
		\[
	 W=\left(
\begin{array}{cc}
\dfrac{\pi\vartheta^2}{rb} & 0 \\
\dfrac{\pi\vartheta^2}{rb\left(1+2\vartheta\right)} & \dfrac{2\pi\vartheta^2}{rb\left(1+2\vartheta\right)}
\end{array}
\right)
	=\left(
\begin{array}{cc}
2\pi / N_{\varphi} & 0 \\
\pi / N_{\psi} & 2\pi / N_{\psi}
\end{array}
\right)
	\]
	We obtain the interlaced sampling scheme $\left(\varphi_k, \psi_{k,m} \right)$ by
\begin{align*}
\varphi_k = \dfrac{k2\pi}{N_{\varphi}},\quad \psi_{k,m} = \dfrac{\pi(k+2m)}{N_{\psi}}.
\end{align*}
Let $l=k+2m$ then we get the efficient sampling scheme $\left(\varphi_k, \psi_{kl} \right)$ as
\begin{align*}
\varphi_k = &\dfrac{k2\pi}{N_{\varphi}} \hspace{1cm} \text{for} \hspace{1cm} 0\leq k \leq N_{\varphi}-1 \\
\alpha_{k,l} = &\dfrac{\pi l}{N_{\psi}} \hspace{1cm} \text{for k, l are parity and} \hspace{1cm} 0 \leq l < \dfrac{N_{\psi}\arcsin(r_0 / r)}{\pi}
\end{align*}
	which the sampling conditions are
\begin{align*}
N_{\varphi}\geq \dfrac{2\pi rb}{\vartheta^2}, \quad N_{\psi} \geq \dfrac{rb\left(1+2\vartheta\right)}{\vartheta^2}.
\end{align*}	
Taking the limit $\vartheta \rightarrow 1 $ we obtain 
\begin{align*}
N_{\varphi}\geq 2\pi rb, \quad N_{\psi} \geq  3rb.
\end{align*}	

\section*{Acknowlegement} Linh Nguyen's research is partially supported by the NSF grants, DMS 1212125 and DMS 1616904.

\end{document}